\documentclass[a4paper,11pt]{amsart}
\usepackage{tikz,tikz-cd,amscd,amsmath, amsthm,amssymb,braket,mathrsfs,graphicx,verbatim,mathtools}
\usepackage{pgfplots}
\usepackage{soul}
\usepackage{hyperref}
\hypersetup{colorlinks=true, linkcolor=black,citecolor=black}
\usepackage[margin=3cm]{geometry}

\usetikzlibrary{matrix,arrows,decorations.pathmorphing}
\usepackage[all]{xy}
\newtheorem{theorem}{Theorem}[]
\newtheorem{proposition}[theorem]{Proposition}
\newtheorem{lemma}[theorem]{Lemma}
\newtheorem{corollary}[theorem]{Corollary}

\newtheorem*{question}{Question}
\theoremstyle{definition}\newtheorem{definition}[theorem]{Definition}
\theoremstyle{definition}
\theoremstyle{remark}\newtheorem{remark}{Remark}
\usepackage{xcolor}
\definecolor{steel}{HTML}{396a93}
\def\l2{{L^2(du)}}
\def\hh1{{H^1(du)}}

\def\SU{\mathcal{U}}

\def\R{\mathbb{R}}
\def\Z{\mathbb{Z}}
\def\S{\mathbb{S}}

\def\M{\mathcal{M}}

\def\I{\mathcal{I}}

\def\SC{\mathcal{C}}
\def\SA{\mathcal{A}}

\newcommand{\sS}{\mathcal{S}}
\newcommand{\N}{\mathbb{N}}
\newcommand{\G}{\mathcal{G}}
\newcommand{\HoneC}{H^1(du)\setminus\SA}
\newcommand{\DQX}{X''\!*\G}
\newcommand{\DAX}{\Ro X'\hspace{-0.5mm}*\G}
\newcommand{\norm}[1]{{\left \Vert {#1}\right \Vert}}

\newcommand{\ip}[1]{{\left \langle {#1} \right \rangle}}
\newcommand{\IPh}[2]{{\left \langle {#1},{#2} \right \rangle_{H^1(du)}}}
\newcommand{\h}[1]{\left\|#1\right\|_{H^1(du)}}
\renewcommand{\l}[1]{\left\|#1\right\|_{L^2(du)}}
\newcommand{\IPl}[2]{{\left \langle {#1},{#2} \right \rangle_{L^2(du)}}}

\newcommand{\abs}[1]{\left | {#1}\right |}

\DeclareMathOperator{\rot}{R}
\DeclareMathOperator{\Ro}{R}
\DeclareMathOperator{\crit}{crit}
\DeclareMathOperator{\ran}{ran}

\begin{document}
\title[Sobolev gradient flow]{A Sobolev gradient flow for the area-normalised Dirichlet energy of $H^1$ maps}

\author[S. Okabe]{Shinya Okabe}
\address{Mathematical Institute\\ Tohoku University\\
Sendai 980-8578\\
Japan}
\email{shinya.okabe@tohoku.ac.jp}

\author[P. Schrader]{Philip Schrader}
\address{
School of Mathematics, Statistics, Chemistry and Physics\\
Murdoch University\\
South St, Murdoch $6150$, WA\\
Australia}
\email{phil.schrader@murdoch.edu.au}

\author[G. Wheeler]{Glen Wheeler}
\address{Institute for Mathematics and its Applications, School of Mathematics and Applied Statistics\\
University of Wollongong\\
Northfields Ave, Wollongong, NSW $2500$\\
Australia}
\email{glenw@uow.edu.au}

\author[V.-M. Wheeler]{Valentina-Mira Wheeler}
\address{Institute for Mathematics and its Applications, School of Mathematics and Applied Statistics\\
University of Wollongong\\
Northfields Ave, Wollongong, NSW $2500$\\
Australia}
\email{vwheeler@uow.edu.au}

\thanks{
	This research was supported in part by
grants
JSPS KAKENHI 
20KK0057, 21H00990 and ARC DP180100431, DE190100379.
}

\subjclass[2020]{53E99, 34C40, 58B20}

\begin{abstract}
In this article we study the $H^1(du)$-gradient flow for the energy $E[X] =
Q[X]/A[X]$ where $Q[X]$ is the Dirichlet energy of $X$, $A[X]$ is the signed
enclosed area of $X$, and $X:\S\rightarrow\R^2$ is a $H^1(du)$ map.
We prove that solutions with initially positive signed enclosed area exist
eternally, and converge as $t\rightarrow\infty$ to a (possibly multiply-covered) circle.
In this way we recover an improved parametrised isoperimetric inequality for
$H^1(du)$ maps.
\end{abstract}

\maketitle


\section{Introduction}

Consider the space $H^1(du)$: maps $X:\S\rightarrow\R^2$ (we use $\S$ to mean
$\R/2\pi\Z$) with distributional derivative $X_u$ (or $X'$) belonging to
$L^2(du)$ endowed with the metric
\[
\IPh{X}{Y} = \IPl{X}{Y} + \IPl{X'}{Y'}
	= \int_0^{2\pi} X(u)\cdot Y(u) + X_u(u)\cdot Y_u(u)\,du
\]
where $X,Y\in H^1(du)$.
On $H^1(du)$ we consider three functionals:
\begin{align*}
\text{Length:}\qquad 
L[X] &= \int_0^{2\pi} |X_u|\,du\,,
\\
\text{Signed area:}\qquad 
A[X] &= -\frac12\int_0^{2\pi} X(u)\cdot\Ro X_u(u)\,du\,,\text{ and}
\\
\text{Dirichlet energy:}\qquad 
Q[X] &= \frac12\int_0^{2\pi} |X_u|^2\,du\,.
\end{align*}
Above we use $\Ro$ to denote the matrix $\text{Rot}{\,}_{\frac\pi2} = \begin{pmatrix}0 & -1 \\ 1 & 0
\end{pmatrix}$ which is rotation through an angle of $\frac\pi2$ in the
counter-clockwise direction.
Each of the functionals $L, A$ and $Q$ are well-defined on $H^1(du)$ (see Section \ref{SNlengthareadiri}).
The quantity $A[X]$ agrees with the measure of the interior region of a map $X\in H^1(du)$ if $X$
is also a Jordan curve.
Otherwise, some connected components of $\R^2\setminus X(\S)$ may count for `negative' area.
For example, if $X$ is a symmetric lemniscate then $A[X] = 0$. 
This is the reason for the name \emph{signed area}.

On the subspace $H^1_+(du)$ of $H^1(du)$ consisting of maps $X$ for which
$A[X]>0$, we consider the energy
\[
	E[X] = \frac{Q[X]}{A[X]} = -\frac{\int_0^{2\pi} |X_u|^2\,du}{\int_0^{2\pi} X\cdot\Ro X_u\,du}
\,.
\]
For $X\in H^1_+(du)$, the energy $E[X]$ is finite and controls the isoperimetric
ratio (see Lemma \ref{LMeiso}).
Indeed, in the regular homotopy class of an $n$-times covered circle, critical
points of $E$ are only the $n$-times covered circles themselves.
It is therefore expected to be isoperimetrically interesting to study the
minimisation of this energy $E$.

In the present paper, our goal is to study an evolution equation that reduces
the energy $E$.
As $E$ involves part of the $H^1(du)$ norm itself, our approach is to
consider the $H^1(du)$-gradient flow of $E$.
This is our main object of study, the evolution equation
\begin{equation}
\label{EQevol}
X_t = -DE[X]
\end{equation}
where now $X:\S\times[0,T)\rightarrow\R^2$ is a one-parameter family of maps
$X(\cdot,t)\in H^1_+(du)$, and by $DE$ we have denoted the $H^1(du)$-gradient
of $E$.

Sobolev gradients have been considered for quite some time in the literature.
We refer the interested reader to \cite{Neuberger:2010aa} for a survey.
Very recently flows of curves \cite{SWW,OS} have been considered, using a
Sobolev Riemannian metric.
In \cite{SWW}, the $H^1(ds)$-gradient flow for length is studied, and
convergence of a rescaling to an asymptotic shape is proved.
However, there is no useful characterisation of this shape.
In particular, this will not yield a geometric inequality (for instance, the
isoperimetric inequality).

This was understood to be due to the curve shrinking too fast and scale
dependence of the gradient.
One natural way to overcome this is to introduce a constraint into the flow.
Here, we take a slightly different approach.
We penalise smallness of the area directly in the energy, and take the flat
$H^1(du)$-gradient.
Having the energy be a quotient penalises very strongly maps where the
denominator vanishes.
A typical consequence of this is control on the denominator, which we are lucky
enough to obtain (see Corollary~\ref{CYmiracle}).
Second, the flat gradient allows us to consider a greater variety of techniques
for the analysis of the flow.
In particular we wish to highlight our use of the Morse--Bott criterion (due to
Feehan~\cite{feehan2020morse}) to establish convergence of the flow.

In this paper we 
 have two main contributions: (1) Theorem \ref{TMmain} below on the flow
itself; and (2) an application of the theorem to isoperimetry (Theorem
\ref{TMparamiso}).

\begin{theorem}
\label{TMmain}
Let $X_0\in H^1_+(du)$.
The $H^1(du)$-gradient flow for $E$ with initial data $X_0$,
$X:\S\times\R\rightarrow\R^2$, exists eternally and converges exponentially fast in $H^1(du)$ to
a multiply-covered circle.
\end{theorem}

Our proof is broken into a number of steps.
\begin{enumerate}
\item We classify equilibria (Proposition~\ref{PNequil}).
\item We show local well-posedness by proving estimates for $E$ and $DE$ and
applying Picard--Lindel\"of (Theorem~\ref{TMSTE}).
\item We show eternal existence by proving estimates for the energy and the
area (Corollary~\ref{CYeternal}).
\item We prove convergence by first showing subconvergence in the $H^1(du)$
topology (Proposition \ref{PNconv2}), and then using the aforementioned
Morse--Bott criterion from Feehan \cite{feehan2020morse} that gives an optimal
\L ojasiewicz--Simon gradient inequality (Theorem~\ref{TMmb} and Corollary
\ref{CYrate}).
\end{enumerate}

This paper is organised as follows.
In the next subsection we give an application of Theorem~\ref{TMmain} to
\emph{parametrised isoperimetry}.
By this we mean the isoperimetric inequality for parametrised immersed curves
that is sharp on multiply-covered circles.
Our setting is further detailed in Section~2.
First variations and calculations of the $H^1(du)$-gradients of the functionals
under consideration are given in Section~3.
All equilibra for the flow are classified in Section~4.
The local well-posedness of the flow is proved in Section~5.
Section~6 is concerned with establishing a-priori estimates, proving eternal
existence, and giving subconvergence.
Full convergence, using the Morse--Bott criterion from Feehan, is the subject of
Section~7.
The argument establishing a parametrised isoperimetric inequality (see Theorem~
\ref{TMparamiso} in the next subsection) is in Section~8, concluding the paper.


\subsection{Application to parametrised isoperimetry}

To each map $X\in H^1(du)$ we may associate a number $\I[X]$,
$\I:H^1(du)\rightarrow[0,\infty]$ called the isoperimetric ratio.
It is defined by
\[
\I[X] = \begin{cases} \frac{L^2[X]}{4\pi |A[X]|}\qquad\text{ if $A[X] \ne 0$}\,, \\ +\infty\quad\ \qquad\text{ otherwise}.\end{cases}
\]
Above we have used $L[X]$ to denote the (arc) length of $X$, and $A$ for the
signed enclosed area as before.

Classically, the isoperimetric inequality states that $\I[X] \ge 1$, with
$\I[X] = 1$ if and only if $X$ is a standard round circle.
We refer the interested reader to the survey \cite{Oss78}.

Naturally, there have been numerous generalisations to the classical isoperimetric inequality.
One that we pursue here is the following.
Considering a singly-covered circle $Y_1$, we are led to conjecturing that $\I[X] \ge \I[Y_1] = 1$.
If we instead begin with an $n$-times covered circle $Y_n$, we find $\I[Y_n] = n$.
Along this line of reasoning, we pose the following general question.

\begin{question}
Fix $n>1$.
What is the largest class of maps $\SC\subset H^1(du)$ such that
\begin{equation}
\label{EQparamiso}
	\I[X] \ge \I[Y_n] = n\text{ for all $X\in\SC$}?
\end{equation}
\end{question}

We call inequalities of the form \eqref{EQparamiso} \emph{parametrised
isoperimetric inequalities}.
Our motivation for this name is that the inequality \eqref{EQparamiso} focuses
on the case of immersed parametrised curves, as opposed to the most general
classical isoperimetric inequalities that approach the question from the
perspective of geometric measure theory.

It seems difficult at first to find an appropriate class of curves in which to
consider $Y_n$ a minimiser of the isoperimetric ratio.
This is because the turning number or total turning angle (divided by $2\pi$)
of the tangent vector seems to be an important notion.
But by incorporating many small loops into a given map $X\in H^1(du)$, which do
not affect the length nor the area very much but nevertheless dramatically
change the turning number, one is quickly led to believing that the turning
number should not play a significant role in the parametrised isoperimetry
question.

On the other hand, the rotational symmetry satisfied by $Y_n$ does play a
significant role.
Indeed, explicit calculation with classical immersed curves such as folia with
$m$-leaves that cover a circle $n$ times (in terms of total angle), shows that
their isoperimetric ratio is greater than that of the corresponding $Y_n$.
In our view, this is decisive.

While a complete answer to the parametrised isoperimetry question remains open,
there exists some work in the literature in this direction.
Partial answers were given by Epstein--Gage~\cite{EG}, Chou~\cite{Ch}, Wang--Li--Chao~\cite{WLC} 
and S\"ussmann~\cite{Suess}.
These works rely on parametrisation by angle, which motivated their
assumption that the curve is locally convex, in addition to rotationally symmetric.
This also necessitates some additional regularity, as locally convex curves are
almost everywhere twice differentiable.
For non-convex curves, Miura--Okabe~\cite{MO} answered the question using a
calculus of variations approach on curves with regularity class $W^{2,1}(du)$.

For us here, we are also able to give a partial answer.
Our proof uses the Sobolev $H^1(du)$-gradient flow for $E$.

\begin{theorem}
\label{TMparamiso}
Suppose $X\in H^1(du)$.
Then
\[
\I[X] \ge 1\,,
\]
with equality if and only if $X$ is a singly-covered circle.

Furthermore, if $X$ is a rotationally symmetric map in the sense that there are
positive integers $n, m$ such that $1\le n \le m$ and
\begin{equation}
\label{EQsymintro}
	X\Big(u+\frac{2\pi}{m}\Big) = \text{Rot}{\,}_{\frac{2\pi n}{m}} X(u)\,,
\end{equation}
then
\begin{equation}
\label{EQparamisoimeqintro}
	\I[X] \ge n\,,
\end{equation}
with equality if and only if $X$ is an $n$-times covered circle. 
\end{theorem}

\begin{remark}
\label{RMKnequalsm}
The condition \eqref{EQsymintro} with $n=m$ is a degenerate case.
If $n=m$, then the image $X(\S)$ is $m$-copies of $X([0,\frac{2\pi}{m}])$.
The leaf $u\mapsto X(mu):\S\rightarrow\R^2$ is in $H^1(du)$ and so if $n=m$ the
flow is equivalent to just flowing a leaf by itself.
Accordingly the conclusion $\I[X]\ge n$ is a consequence of the classical
isoperimetric inequality.
Alternatively, we may flow directly the multiply-covered configuration to a
limit $Y_n$ which we prove is in general one of a family of parametrisations of
possibly multiply-covered circles.
The strong convergence and uniqueness implies that the limit $Y_n$ must, at the
parametrisation level, satisfy \eqref{EQsymintro}.
This implies that it is at least $n$-times covered and so in particular
$\I[Y_n]\ge n$.
\end{remark}

\begin{remark}
We may allow $n>m$ in the theorem if we replace the conclusion
\eqref{EQparamisoimeqintro} with $\I[X] \ge n_0$ where $n_0$ is the least
non-negative residue of $n$ modulo $m$.
The reason for this is as follows.
Since $\text{Rot}{\,}_{2k\pi + \theta} = \text{Rot}{\,}_{\theta}$, if $n>m$
were allowed then the condition \eqref{EQsymintro} would be equivalent to
\begin{equation*}
	X\Big(u+\frac{2\pi}{m}\Big) = \text{Rot}{\,}_{\frac{2\pi n_0}{m}} X(u)\,.
\end{equation*}
Then, the isoperimetric inequality obtained from this is $\I[X] \ge n_0$ (and
not $\I[X] \ge n$).
\end{remark}

\begin{remark}
Our parametrised isoperimetric inequality does not need the turning number
($\omega = \int_0^L k\,ds$) to be defined and indeed the turning number plays no
role.
We note that it may be possible to reduce the regularity requirement from
\cite{MO} and remove the need for a turning number by using an approximation
argument.
In this way, a statement equivalent to ours obtained here using a flow approach
might be obtained.
\end{remark}

The choice of $n$ gives us a rotation action (via \eqref{EQsymintro}) that, due
to strong convergence and uniqueness, will be preserved into the limit.
This gives us a symmetry condition that any candidate limit for the flow with
initial data satisfying \eqref{EQsymintro} must also satisfy.
This implies that the limit must be at least $n$-times covered, and so
(unless $n=1$) the lower bound for the isoperimetric ratio is improved.

We wish to give a constructive example using a given leaf to generate four
parametrisations, each satisfying the rotational symmetry assumption
\eqref{EQsymintro} for a different value of $n$.
Take $X:[0,\frac{2\pi}{m}]\rightarrow\R^2$ to be a curve of class $H^1(du)$
with $X(0) = X(2\pi/m) = 0$, for instance a single petal in an
$n$-folium.
In general, each choice of $n\in\{1,\ldots,m\}$ will yield a different
parametrisation, although some will be multiply-covered and so not
geometrically distinct if $n$ is not relatively prime to $m$.

Let us consider the particular case of $m=4$. There are four subcases.

\begin{itemize}
\item[$(n=1)$] In this case we will be parametrising a quadrifolium in a
non-standard way, by moving from the first petal, to the second, to the third
and finally to the fourth in a counterclockwise loop.
This symmetry does not force the limit to be any more than just once-covered,
which is not an improvement over the classical isoperimetric inequality.
\item[$(n=2)$] In this case the curve $X(\S)$ has a figure-8 shape, with two leaves each covered twice.
The limit parametrisation must be at least twice-covered, and we find $\I[X] \ge 2$.
\item[$(n=3)$] This is the case of the classical quadrifolium.
From this choice the parametrisation begins in the first leaf, then moves to
the fourth leaf, then the third, and finally the second.
The limit parametrisation must be at least thrice-covered, and we find $\I[X] \ge 3$.
\item[$(n=4)$]
In this case, similar to $n=2$, the parametrisation covers the single given
leaf four times.
The limit parametrisation must be at least four-times-covered, and we find
$\I[X] \ge 4$ (cf. Remark \ref{RMKnequalsm}).
\end{itemize}



\subsection*{Acknowledgements}

The authors are grateful for the support provided from JSPS KAKENHI Grants
20KK0057 and 21H00990 to facilitate travel to Tohoku University where the
majority of this research was completed.

The fourth author acknowledges support from ARC Discovery Project DP180100431
and ARC  DECRA DE190100379.

\section{Setting}

In this paper our evolving objects are maps $X\in H^1(du)$, by which we mean the standard notion of maps $X:\S\rightarrow\R^2$ with distributional derivative $X_u$ (or $X'$) belonging to $L^2(\S;\R^2)$.
Here we use $\S$ to denote the standard unit circle, with model $\R/2\pi\Z$, and $\R^2$ is the standard Euclidean plane.

The norm $||X||_{H^1(du)}$ is
\[
	||X||_{H^1(du)}^2 = \int |X|^2\,du + \int |X_u|^2\,du\,.	
\]
Single vertical bars denote the norm of a vector in the plane.
Integrals without limits are always integrated around the entire domain, that
is, from $0$ to $2\pi$.
Once we equip $H^1(du)$ with the inner product
\[
	\IPh{V}{W} = \IPl{V}{W} + \IPl{V_u}{W_u}
	= \int V\cdot W\,du + \int V_u\cdot W_u\,du
\]
the space $H^1(du)$ becomes Hilbert.

\subsection{Length, area, and the Dirichlet energy.}
\label{SNlengthareadiri}

We define three functionals on $H^1(du)$.
First, the length $L:H^1(du)\rightarrow[0,\infty)$ is defined by
\[
	L[X] = \int |X_u|\,du\,.
\]
Since the H\"older inequality implies $L[X] \le \sqrt{2\pi}||X||_{H^1(du)}$ the length functional is well-defined on $H^1(du)$.

Second, the (signed) area $A:H^1(du)\rightarrow\R$ is defined by
\[
	A[X] = -\frac12\int X\cdot \Ro X_u\,du\,.
\]
Above we used the notation $\Ro (x,y) = (-y,x)$ to denote the rotation of a vector counter-clockwise about the origin through an angle of $\frac\pi2$.
By the Cauchy--Schwartz, then H\"older, then Cauchy inequalities, we find
\begin{equation}
\label{EQareawelldef}
A[X] \le ||X||_{L^2(du)}\,||X_u||_{L^2(du)} \le \frac12 ||X||_{H^1(du)}^2
\,.
\end{equation}
This implies that the area functional is also well-defined on $H^1(du)$.

Third, the Dirichlet energy $Q:H^1(du)\rightarrow\R$ is defined by
\[
	Q[X] = \frac12\int |X_u|^2\,du\,.
\]
The functional $Q$ is well-defined on $H^1(du)$; indeed, it is part of the norm:
\[
	||X||_{H^1(du)}^2 = ||X||_{L^2(du)}^2 + 2Q[X]\,.
\]
We note the following lemma.

\begin{lemma}
Suppose $X\in H^1(du)$ and $C = \frac{1}{2\pi}\int X\,du = \overline{X}$.
The $H^1(du)$ norm of $X$ is controlled by the Euclidean length of $C$ and $Q[X]$; in particular,
\[
	||X||_{H^1(du)}^2 \le 4\pi|C|^2 + 6Q[X]\,.
\]
\end{lemma}
\begin{proof}
Note that $C$ is well-defined as $|C| \le \frac{||X||_{L^2(du)}}{\sqrt{2\pi}} \le \frac1{8\pi} + ||X||^2_{H^1(du)}$, where we used the Cauchy inequality.
Now observe the estimate
\[
||X-C||_{L^2(du)}^2 \le ||X_u||_{L^2(du)}^2
\]
which holds by applying the Poincar\'e inequality on $\S$ to each component function of $X-C$.
Thus
\begin{align*}
	||X||_{H^1(du)}^2 
	&= ||(X-C) + C||_{L^2(du)}^2 + ||X_u||_{L^2(du)}^2
\\
	&\le 2||X-C||_{L^2(du)}^2 + 4\pi|C|^2 + 2Q[X]
\\
	&\le 4\pi|C|^2 + 6Q[X]
\end{align*}
as required.
\end{proof}


\section{First variations and $H^1(du)$-gradients}

Let us consider now the energy
\[
	E[X] = \frac{Q[X]}{A[X]} = -\frac{\int |X_u|^2\,du}{\int X\cdot RX_u\,du}
	\,.
\]
The energy is not finite on all of $H^1(du)$ because it achieves the value
$+\infty$ on any non-trivial $H^1(du)$ map $X$ with $A[X]=0$.
Therefore we set $\SA = \{X\in H^1(du)\,:\, A[X] = 0\}$ and consider
$E:\HoneC\rightarrow[0,\infty)$.

We now introduce a key object: the 
fundamental solution to the differential operator $\partial_u^2 - 1$ on $\S$.
We use the notation $\G(u,\hat u)$ for this Greens function.
It has explicit expression given by
\[
	\G(u,\hat u) = -\frac{\cosh(|u-\hat u| - \pi)}{2\sinh\pi}\,.
\]
Apart from satisfying the functional equation, it has the expected symmetry and
regularity properties, along with the estimates $|\G_u| = |\G_{\hat u}| \le
\frac12$ and $|\G| \le \frac12\coth\pi$.

We also note the following fundamental identity.

\begin{lemma}
\label{LMfundiden}
For all $V \in H^1(du)$ and $W\in L^2(du)$ we have
\begin{equation*}
\IPh{V}{W*\G}
	= -\IPl{V}{W}\,,
\end{equation*}
and
\begin{equation*}
\IPh{V}{W_u*\G}
	= \IPl{V_u}{W}\,,
\end{equation*}
\end{lemma}
\begin{proof}
First note that the functional equation $\G_{uu} - \G = \delta$ implies
$W*\G\in H^2(du)$ and $(W*\G)_{uu} = W + W*\G$.
Using this, the symmetry of $\G$ and integration by parts, we find
\begin{align*}
\IPh{V}{W*\G}
	&= \IPl{V}{W*\G} + \IPl{V_u}{W*\G_u}
\\
	&= \IPl{V}{W*\G} - \IPl{V}{W + W*\G}
	= -\IPl{V}{W}\,.
\end{align*}
For the second identity, we use a similar argument, noting that $(W_u*\G)_u = W + W*\G$ and so
\begin{align*}
\IPh{V}{W_u*\G}
	&= \IPl{V}{W_u*\G} + \IPl{V_u}{(W_u*\G)_u}
\\
	&= \IPl{V}{(W*\G)_u} + \IPl{V_u}{W + W*\G}
	=  \IPl{V_u}{W}\,,
\end{align*}
as required.
\end{proof}

We shall use the notation $DQ$, $DA$ and $DE$ for the $H^1(du)$-gradient of the Dirichlet energy, area and the energy $E$ respectively.
We use a lowercase $d$ to denote the derivative.

First we calculate the $H^1(du)$-gradient of $Q$ and $A$.

\begin{lemma}
\label{LMqvar}
The $H^1(du)$-gradient of $Q$ at $X\in H^1(du)$ is
\[
DQ[X] = \DQX\,.
\]
\end{lemma}
\begin{proof}
Let $\hat X = X + tV$ where $t>0$ and $X, V\in H^1(du)$.
We calculate (a.e.)
\[
	\frac{d}{dt}|\hat X_u|^2
	= 2|\hat X_u|\,\frac{d}{dt}|\hat X_u|
	= 2(\hat X_u\cdot V_u)
\,.
\]
Thus
\begin{align}
\label{EQphil1}
	dQ[X](V) = \IPl{X_u}{V_u}
\,.
\end{align}
Furthermore the $H^1(du)$-gradient of $Q$ at $X$, $DQ[X]\in H^1(du)$, is defined by
\begin{equation}
\label{EQphil2}
  dQ[X](V) = \IPh{DQ[X]}{V} = \IPl{-(DQ[X])_{uu} + DQ[X]}{V}\,,
\end{equation}
where by $(DQ[X])_{uu}$ we mean the second distributional derivative of $DQ[X]$, which is defined when paired with $V\in H^1(du)$ (by the equality above).
Differentiating distributionally, we have
\[
	-\IPl{X_{uu}}{V} = \IPh{DQ[X]}{V}
	\,.
\]
Applying the fundamental identity (Lemma \ref{LMfundiden}) we see that $DQ[X] = X_{uu}*\G$, as required.

\end{proof}

\begin{lemma}
\label{LMavar}
The $H^1(du)$-gradient of $A$ at $X\in H^1(du)$ is
\[
DA[X] = \DAX\,.
\]
\end{lemma}
\begin{proof}
Using the same notation as above we find
\begin{align*}
	\frac{d}{dt} \Big(\hat X\cdot \Ro \hat X_u\Big)
	&=  V\cdot \Ro \hat X_u
	 + \hat X\cdot \Ro V_u
\\
	&=  V\cdot \Ro \hat X_u
	 - \hat X_u \cdot \Ro V
	 + (\hat X\cdot \Ro V)_u
\\
	&=  2V\cdot \Ro \hat X_u
	 + (\hat X\cdot \Ro V)_u
\end{align*}
and so
\[
	dA[X](V) = -\IPl{\Ro X_u}{V}
\,.
\]
The $H^1(du)$-gradient of $A$ at $X$, $DA[X]\in H^1(du)$ is defined by
\[
  dA[X](V) = \IPh{DA[X]}{V}\,. 
\]
That is,
\[
	-\IPl{\Ro X_{u}}{V} = \IPh{DA[X]}{V}
\]
so applying the fundamental identity (Lemma \ref{LMfundiden}) yields that $DA[X] = \Ro X'*\G$, as required.

%
\end{proof}

\begin{lemma}
\label{LMevar}
The $H^1(du)$-gradient of $E$ at $X\in\HoneC$ is
\[
DE[X] = A^{-1}[X](X'' - E[X]\Ro X')*\G
\,.
\]
\end{lemma}
\begin{proof}
Since $Q$ and $A$ are differentiable on $\HoneC$ we use Lemmata~\ref{LMqvar} and~\ref{LMavar} to calculate
\begin{align*}
	dE[X](V) 
	&= A^{-1}[X]dQ[X](V) - A^{-2}[X]Q[X]dA[X](V)
\\
	&= A^{-1}[X]\IPl{X_u}{V_u} - A^{-2}[X]Q[X]\IPl{-\Ro X_u}{V}
\\
	&= A^{-1}[X]\IPh{\DQX}{V} - A^{-1}[X]E[X]\IPh{\DAX}{V}
\\
	&= \IPh{A^{-1}[X](\DQX - E[X]\DAX)}{V}
	\,,
\end{align*}
so 
\[
DE[X] = A^{-1}[X](\DQX - E[X]\DAX)
 = A^{-1}[X](X'' - E[X]\Ro X')*\G
\,,
\]
as required.
\end{proof}

Finally let us note the following relationship between our energy $E$ and the isoperimetric ratio.

\begin{lemma}
For $X\in H^1(du)$ we have
\[
	E[X] \ge \I[X]\,.
\]
\label{LMeiso}
\end{lemma}
\begin{proof}
This is just the H\"older inequality:
\[
	L^2[X] = ||X_u||_{L^1(du)}^2 \le 2\pi||X_u||_{L^2(du)}^2 = 4\pi Q[X]
	\,.
\]
\end{proof}

\begin{remark}
We remark that although $DQ$ and $DA$ are linear, $DE$ is not.
Indeed, both $DE[aX] \ne a\,DE[X]$ and $DE[X+Y] \ne DE[X] + DE[Y]$.
The gradient $DE$ is homogeneous of degree $-1$, which implies that there is a space-time scaling that maps solutions to solutions for the flow.
This is natural; the degree of homogeneity for the gradient of an $\alpha$-homogeneous functional is $\alpha-1$.
So $DQ$ and $DA$ are $1$-homogeneous, and $DE$ is $(-1)$-homogeneous.
This space-time rescaling may have application to simulations, and to understanding more exotic evolutions than those discussed here (for instance, curves with initially negative energy).
\end{remark}


\section{Equilibria for the gradient flow}

A $H^1(du)$-gradient flow for $E[X]$ is a one-parameter family of
time-differentiable maps $X:\S\times[0,T)\rightarrow\R^2$ with $X(\cdot,t)\in
\HoneC$ satisfying
\begin{equation}
\label{EQflow}
	X_t = -DE[X]
	= -A^{-1}[X](X'' - E[X]\Ro X')*\G
	\,.
\end{equation}
This is a non-linear ordinary differential equation.

\begin{definition}
If $X$ is a one-parameter family of time-differentiable maps in $\HoneC$
satisfying \eqref{EQflow}, then we say that  $X$ is a {\it $H^1(du)$-gradient
flow for the area-normalised Dirichlet energy} or  a {\it $H^1(du)$-gradient
flow for $E$}.
\end{definition}

The equilibrium set for the $H^1(du)$-gradient flow for $E$ is
finite-dimensional and can furthermore be completely described.

\begin{proposition}
\label{PNequil}
The set of stationary solutions to the $H^1(du)$-gradient flow for $E$ is
\[
	\sS = \bigg\{Y:\S\rightarrow\R^2\,,
		Y(u)
	= \frac1\ell\begin{pmatrix} a\sin\ell u + b\cos\ell u - b\\ -a\cos\ell u + b\sin\ell u + a \end{pmatrix} + \begin{pmatrix} c \\ d \end{pmatrix}
, a,b,c,d\in\R, \ell\in\N\bigg\}
\,.
\]
Note that $\sS$ is a countable union of four-dimensional spaces.
\end{proposition}

\begin{remark}
Proposition \ref{PNequil} implies quantisation of the energy to the natural numbers.
This means that the final energy level of the flow (assuming convergence)
starting at $X_0$ is in the set $\N\cap [1,E[X_0]]$.
\end{remark}

Qualitatively the equilibrium set consists of multiply-covered circles.
However the major advantage that this description does not make clear is that
the circles may choose their parametrisation from a finite-dimensional set.
We do not suffer from parametrisation invariance, and this makes the dynamics
of our flow more straightforward.

\begin{proof}
Suppose we have an equilibrium solution to \eqref{EQflow}, that is, a map $Y \in \HoneC$ satisfying
\[
	DE[Y]
	= A^{-1}[Y](Y'' - E[Y]\Ro Y')*\G = 0
\,.
\]
First, $|A[Y]| < \infty$ so
\begin{equation*}
	(Y'' - E[Y]\Ro Y')*\G = 0
	\,.
\end{equation*}
Using the functional equation for $\G$, we find
\begin{equation}
\label{EQphil3}
	Y = -Y*\G - E[Y]\Ro Y'*\G
	  = -(Y + E[Y]\Ro Y')*\G
	\,.
\end{equation}
Now differentiating (distributionally) and using the functional equation for $\G$ again, we have
\begin{equation}
\label{EQphil4}
	Y_{uu} = -Y - E[Y]\Ro Y_u - (Y + E[Y]\Ro Y')*\G
	\,.
\end{equation}
This shows that $Y_{uu}\in L^2(du)$; in fact, we may continue to gain
regularity without limit in this way.
Thus $Y$ is smooth and (substituting \eqref{EQphil3} into \eqref{EQphil4}) we have
\begin{equation}
\label{EQequil}
	Y_{uu} = E[Y]\Ro Y_u
	\,.
\end{equation}
From this equation we can learn several things about the equilibrium $Y$.
Suppose $Y(u) = (x(u), y(u))$ and set $\ell=E[Y]$; then
\[
	\frac{d}{du}
	\begin{pmatrix}
		x'(u)
\\
		y'(u)
	\end{pmatrix}
	=
	\begin{pmatrix}
		0 & -\ell
\\
		\ell & 0
	\end{pmatrix}
	\begin{pmatrix}
		x'(u)
\\
		y'(u)
	\end{pmatrix}
\]
Thus we may solve for the components as
\begin{equation*}
	\begin{pmatrix}
		x'(u)
\\
		y'(u)
	\end{pmatrix}
	=
	\text{exp}
	\begin{pmatrix}
		0 & -\ell u
    \\
		\ell u & 0
	\end{pmatrix}
	\begin{pmatrix}
		a
\\
		b
	\end{pmatrix}
	= 
	\begin{pmatrix}
		\cos(\ell u) & -\sin(\ell u)
\\
		\sin(\ell u) & \cos(\ell u)
	\end{pmatrix}
	\begin{pmatrix}
		a
\\
		b
	\end{pmatrix}
	= 
	\begin{pmatrix}
		a\cos(\ell u) - b\sin(\ell u)
\\
		a\sin(\ell u) + b\cos(\ell u)
	\end{pmatrix}
\end{equation*}
where $a,b$ are such that $Y'(0) = (a,b)$.
Note that this equation can also be written as
\[
	Y'(u) = \text{Rot}{\,}_{\ell u} \begin{pmatrix} a \\ b \end{pmatrix}
	\,.
\]
%
%
%
%
%
%
Since $\frac{d}{d\theta}\text{Rot}{\,}_{\theta} = \text{Rot}{\,}_{\theta+\frac\pi2}$,
we find
\[
	\bigg(\ell^{-1}\text{Rot}{\,}_{\ell u - \frac\pi2}\begin{pmatrix} a \\ b \end{pmatrix}\bigg)'
	= \text{Rot}{\,}_{\ell u} \begin{pmatrix} a \\ b \end{pmatrix}
\,.
\]
This implies (assuming that $Y(0) = (c,d)$) that
\begin{align*}
	Y(u) - \begin{pmatrix} c \\ d \end{pmatrix}
	&= \ell^{-1}\text{Rot}{\,}_{\ell u - \frac\pi2}\begin{pmatrix} a \\ b \end{pmatrix}
	- \ell^{-1}\text{Rot}{\,}_{ - \frac\pi2}\begin{pmatrix} a \\ b \end{pmatrix}
	= \ell^{-1}\begin{pmatrix} a\sin\ell u + b\cos\ell u \\ -a\cos\ell u + b\sin\ell u \end{pmatrix}
	- \ell^{-1}\begin{pmatrix} b \\ -a \end{pmatrix}\intertext{so}
	Y(u)
	&= \ell^{-1}\begin{pmatrix} a\sin\ell u + b\cos\ell u - b\\ -a\cos\ell u + b\sin\ell u + a \end{pmatrix} + \begin{pmatrix} c \\ d \end{pmatrix}
\,,
\end{align*}
as claimed.

Periodicity implies that $\ell = E[Y]$ is integer-valued, and $\ell \ne 0$ because $Y\not\in \SA$.
Furthermore, explicit calculation shows that
\[
	A[Y] = \frac{\pi}{\ell}(a^2 + b^2)
\,.
\]
In particular, the area of $Y$ and $\ell$ share the same sign.
Since the energy decreases (to negative $\infty$) by homothetically shrinking
any curve with negative area to a point, there do not exist equilibria with
negative area.
Therefore the area of $Y$ must be positive and so $\ell$ must be a positive integer.
\end{proof}

\begin{remark}
We did not use the isoperimetric inequality in the proof of
Proposition \ref{PNequil}.
\end{remark}



\section{Existence and uniqueness}

To show that a unique solution to \eqref{EQflow} exists, we use the Picard--Lindel\"of theorem.
We will not be able to study the flow where the initial map has zero area.
For such maps $X$, the energy $E[X]$ is unbounded (unless they are points, i.e.
constant maps) and so is the gradient $DE[X]$.
Furthermore, when the area of $X$ is negative, the energy $E[X]<0$ and the
descent flow will send $E[X(\cdot,t)]\rightarrow-\infty$.
This does not seem helpful from the perspective of isoperimetry, although
mathematically it does appear interesting.

Thus we subtract maps with non-positive area from the space $H^1(du)$, and seek
to prove existence on the space $H_+^1(du)$ where
\[
	H^1_+(du) = \{X\in H^1(du)\,:\,A[X] > 0\}\,.
\]
Note that $\HoneC \supset H^1_+(du)$, so the work in the previous sections
applies here.

The isoperimetric inequality is trivial for $X\in\SA$ (when the area vanishes),
and length is invariant under change of orientation (reparametrisation sending
$u\mapsto 2\pi-u$) whereas area changes sign.
So for the practical purpose of establishing parametrised isoperimetric
inequalities for all $X\in H^1(du)$ (Theorem~\ref{TMparamiso}), considering the
$H^1(du)$-gradient flow only on the space $H_+^1(du)$ is not a significant
restriction.

We require the following lemma.

\begin{lemma}\label{blip}
Given 
$X_0\in H^1_+(du)$ let
\begin{align*}
Q_b
 &= \bigg\{X\in H^1_+(du)\ :\  
\\
&\qquad\qquad {\Vert X-X_0\Vert}_\hh1^2 \leq b^2 
		\le \min\Big\{\frac12\h{X_0}^2, \frac{A^2[X_0]}{8||X_0||_{H^1(du)}^2}\Big\} \bigg\}
\end{align*}
where $b>0$ is fixed. 
Then there exist constants $J\geq 0$ and $K>0$ depending only on $b$ and $\h{X_0}$ such that 
\begin{align*}
	\norm{DE[X]}_\hh1  &< K, && \text{for all} \,\, X\in Q_b\, , \\
\norm{DE[X] - DE[Y]}_\hh1  &\leq J \norm{X-Y}_\hh1, && \text{for all}\, \, X,Y\in Q_b\,.
\end{align*}
\end{lemma}
\begin{proof}
To obtain the estimates and remain away from the problematic set (maps with
non-positive area) it is necessary that $A[X]$ for each $X\in Q_b$ is bounded
away from zero.
This is the reason for the upper bound on $b$.
For $X\in Q_b$
 we calculate
\begin{align*}
	|A[X] - A[X_0]|
	&= \bigg|-\frac12 \int X\cdot\Ro X_u - X_0\cdot \Ro (X_0)_u\,du \bigg|
\\
	&= \bigg|-\frac12 \int (X-X_0)\cdot\Ro X_u + (\Ro X_u-\Ro (X_0)_u)\cdot X_0\,du \bigg|
\\
	&\le \frac12 ||X-X_0||_{L^2(du)} ||X_u||_{L^2(du)}
\\&\quad
		 + \frac12 ||X_u-(X_0)_u||_{L^2(du)} ||X_0||_{L^2(du)}
\end{align*}
so
\begin{align*}
2|A[X]-A[X_0]|^2 
&\le \h{X-X_0}^2(\h{X}^2+\h{X_0}^2)
\\
&\le \h{X-X_0}^2(2\h{X-X_0}^2 + 3\h{X_0}^2)
\\
&\le b^2(2b^2 + 3\h{X_0}^2)
\,.
\end{align*}
Using the upper bound on $b$, we find
\begin{align*}
	|A[X] - A[X_0]|^2
&\le \frac18\frac{A^2[X_0]}{\h{X_0}^2}(\frac{1}{2}\h{X_0}^2 + \frac{3}{2}\h{X_0}^2)
\\
&= \frac{A^2[X_0]}{4}
\,.
\end{align*}
The final estimate is
\begin{equation}
\label{EQareadiff}
\frac12A[X_0] \le A[X] < \frac32A[X_0]
\,,
\end{equation}
for all $X\in Q_b$.

Recall that $\G$ is given by
\[
	\G(u,\hat u) = -\frac{\cosh(|u-\hat u| - \pi)}{2\sinh\pi}\,.
\]
Direct calculation yields $\G(u,\hat u) = -\G(\hat u, u)$, $|\G(u,\hat u)| \le
\frac12\coth\pi$ and $|\G_u(u,\hat u)| \le \frac12$.
Using Lemma \ref{LMfundiden}, we have
\begin{align*}
\h{X*\G}^2
 &= \IPh{X*\G}{X*\G}
 = \IPl{X*\G}{-X}
\\
 &\le \frac12\h{X*\G}^2 + \frac12\l{X}^2
\end{align*}
so
\begin{equation}
\label{EQh1tol2}
\h{X*\G}^2
\le \l{X}^2\,.
\end{equation}

The gradient is
\[
	DE[X] = A^{-1}[X](X'' - E[X]\Ro X')*\G\,.
\]
Using the fundamental solution property of $\G$ we find
\[
	DE[X] = \frac{X}{A[X]} + A^{-1}[X](X - E[X]\Ro X')*\G\,.
\]
Thus, applying \eqref{EQareadiff}, \eqref{EQh1tol2} and using the 
bound for $|\G_u|$, we estimate
\begin{align*}
\h{DE[X]}
	&\le 
	\frac2{A[X_0]}\Big(
		\h{X} +
		\h{(X - E[X]\Ro X_u)*\G}
	\Big)
\\
	&\le 
	\frac2{A[X_0]}\Big(
		\h{X} + \h{X*\G} +
		E[X]\h{\Ro X*\G_u}
	\Big)
\\
	&\le 
	\frac2{A[X_0]}\Big(
		2\h{X} +
		E[X]\h{\Ro X*\G_u}
	\Big)
\\
	&\le 
	\frac{(4+E[X])\h{X}}{A[X_0]}
\,.
\end{align*}
Now
\begin{equation}
\label{EQinqbh1est}
\h{X} \le \h{X-X_0} + \h{X_0} \le b + \h{X_0}\,,
\end{equation}
so we have the energy bound
\begin{equation}
\label{EQenergyonq}
E[X] = \frac{Q[X]}{A[X]} \le \frac{\frac12\h{X}^2}{\frac12 A[X_0]}
 \le \frac{(b + \h{X_0})^2}{A[X_0]}
\,.
\end{equation}
Thus (using \eqref{EQinqbh1est}, \eqref{EQenergyonq}) we find
\begin{equation}
\label{fboundh1}
\h{DE[X]} \le 
	\bigg(4+\frac{(b+\h{X_0})^2}{A[X_0]}\bigg)\frac{b + \h{X_0}}{A[X_0]}
	=: K
\end{equation}
for all $X\in Q_b$. 
 
It thus remains to establish the Lipschitz estimate for $DE$.
We will use the fact that Lipschitz functions
form an associative algebra and show that individually the area functional $A$,
the energy $E$, and then the maps $X\mapsto X''*\G$, $X\mapsto \Ro X' *\G$ are all
uniformly Lipschitz (in the $H^1(du)$ metric space) on $Q_b$.
This will imply (together with the uniform boundedness of each of $A$, $E$, and
the map $X\mapsto (X'' - \Ro X')*\G$) that $DE$ is Lipschitz on $Q_b$ with
a uniform constant.

{\it Uniform boundedness on $Q_b$.}
The area is uniformly bounded on $Q_b$ from \eqref{EQareawelldef} and
\eqref{EQinqbh1est}, and similarly the energy is clearly uniformly bounded by
\eqref{EQenergyonq}.
A calculation similar to the one used to prove \eqref{fboundh1} establishes
uniform boundedness also of the maps
 $X\mapsto X''*\G$,
 $X\mapsto \Ro X' *\G$; indeed we have the estimates
\begin{align*}
 \h{X''*\G}
&= \h{X + X*\G}
\\
&\le \h{X} + \h{X*\G}
\\
&\le \h{X} + \l{X}
\le 2\h{X} 
\\
&\le 2(b + \h{X_0})
\end{align*}
and
\begin{align*}
 \h{\Ro X' *\G}
&\le \l{\Ro X_u} = \l{X_u}
\\
&\le \h{X} \le b + \h{X_0}
\,.
\end{align*}

{\it Area is Lipschitz on $Q_b$.} This step is completed already, essentially,
see the proof of the estimate \eqref{EQareadiff}.
We find
\begin{align*}
\Big| A[X] - A[Y] \Big|^2
	&\le \frac12\h{X-Y}^2(\h{X}^2 + \h{Y}^2)
\\
	&\le \h{X-Y}^2(b + \h{X_0})^2
\,.
\end{align*}
So on $Q_b$ the area functional is uniformly Lipschitz with constant $(b+\h{X_0})$.

{\it Dirichlet energy.} We calculate
\begin{align*}
\Big|Q[X] - Q[Y]\Big|
	&\le \frac12
	        \bigg|\int |X_u|^2-|Y_u|^2\,du\bigg|
	 = \frac12
	        \bigg|\int (X_u-Y_u)\cdot(X_u+Y_u)\,du\bigg|
\\
	&\le \frac12||X_u-Y_u||_{L^2(du)}\l{X_u + Y_u}
\\
	&\le \frac12||X_u-Y_u||_{L^2(du)}(\h{X} + \h{Y})
\\
	&\le ||X_u-Y_u||_{L^2(du)}(b+\h{X_0})
\,.
\end{align*}
So $Q$ is uniformly Lipschitz with constant $(b+\h{X_0})$.

{\it Energy.} As $E[X] = Q[X]/A[X]$ and $A[X] \ge \frac12A[X_0]$ uniformly, we
also have that $E$ is uniformly Lipschitz.

{\it The maps $X\mapsto X''*\G$ and $X\mapsto \Ro X'*\G$.}
These are bounded linear maps and so they are automatically Lipschitz.

{\it Lipschitz property for $DE$.}
Therefore on $Q_b$, as $DE$ consists of products of uniformly Lipschitz,
bounded maps, it is itself also a uniformly Lipschitz bounded map on $Q_b$.
\end{proof}

Now the theorem follows.

\begin{theorem}\label{TMSTE}
For each $X_0\in H^1_+(du)$ there exists a $T_0>0$ and a unique
map $X\in C^1([-T_0,T_0]; H^1_+(du))$ such that $X(\cdot,0) = X_0$ and $X$ is a
$H^1(du)$-gradient flow for $E$.
\end{theorem}
\begin{proof}
Using \cite[Theorem 3.A]{Zeidler:1986aa}, the estimates in Lemma \ref{blip}
guarantee existence and uniqueness of a solution on the interval provided
 $KT_0 < b$.
\end{proof}



\section{Global control}

In this section we show how a number of remarkable coincidences come together
to give powerful control over any $H^1(du)$-gradient flow for $E$ on
$H^1_+(du)$.


\begin{proposition}
\label{PNmiracle}
Let $X:\S\times[0,T)\rightarrow\R^2$ be a $H^1(du)$-gradient flow for $E$ on
$H^1_+(du)$.
Then $||X(\cdot,t)||_{H^1(du)} = ||X(\cdot,0)||_{H^1(du)}$.
\end{proposition}
\begin{proof}
We calculate:
\begin{align*}
\frac{d}{dt} ||X(\cdot,t)||_{H^1(du)}^2
	&= 2\IPh{X}{X_t}
	 = -2\IPh{X}{DE[X]}
\\
	&= -2A^{-1}[X] \IPh{X}{( X'' - E[X]\Ro X')*\G}
\\
	&= -2A^{-1}[X] \IPh{X}{X''*\G}
	   +2E[X]A^{-1}[X] \IPh{X}{\Ro X'*\G}
\,.
\end{align*}
Now applying Lemma \ref{LMfundiden} to both terms, we find:
\begin{align*}
\frac{d}{dt} ||X(\cdot,t)||_{H^1(du)}^2
	&= -2A^{-1}[X] \IPl{X_u}{X_{u}}
	  - 2\frac{E[X]}{A[X]} \IPl{X}{\Ro X_u}
\,.
\end{align*}
The inner product in the first term is twice the Dirichlet energy
$Q[X]$ and the inner product in the second term is $-2A[X]$.
This allows us to simplify:
\begin{align*}
\frac{d}{dt} ||X(\cdot,t)||_{H^1(du)}^2
	&= -4\frac{Q[X]}{A[X]}
	  - 2\frac{Q[X]}{A^2[X]} (-2A[X])
	= 0\,.
\end{align*}
\end{proof}

Thus the flow is constrained to an infinite dimensional sphere in $H^1(du)$.
The next proposition shows that the centre of the flow is preserved.
Below we use an overline to denote the average, so
\[
\overline{X}(t) = \frac1{2\pi}\int X\,du
\,.
\]

\begin{proposition}
\label{PNmiracle2}
Let $X:\S\times[0,T)\rightarrow\R^2$ be a $H^1(du)$-gradient flow for $E$ on
$H^1_+(du)$.
Then $\overline{X}(t) = \overline{X}(0)$.
\end{proposition}
\begin{proof}
We calculate using integration by parts and $\G_{\hat u} = -\G_u$
\begin{align*}
\frac{d}{dt}\overline{X}(t)
	&= -\frac1{2\pi} \int DE[X(\cdot,t)]\,du
\\
	&= -\frac1{2\pi}
		\int -A^{-1}[X] ( X'' - E[X]\Ro X')*\G\,du
\\
	&= \frac1{2\pi A[X]}
		\int ((X' - E[X]\Ro X)*\G)_u\,du
	= 0\,.
\end{align*}
\end{proof}

Using these we obtain the following.

\begin{proposition}
\label{PNmiracle3}
Let $X:\S\times[0,T)\rightarrow\R^2$ be a $H^1(du)$-gradient flow for $E$ on
$H^1_+(du)$.
Then $||X(\cdot,t)-\overline{X}(t)||_{H^1(du)} =
||X(\cdot,0)-\overline{X}(0)||_{H^1(du)}$.
\end{proposition}
\begin{proof}
Using the previous results we explicitly calculate
\begin{align*}
\frac{d}{dt}&||X(\cdot,t)-\overline{X}(t)||_{H^1(du)}^2
\\
	&= -2\IPh{X-\overline{X}}{DE[X]}
\\
	&= -2A^{-1}[X] \IPh{X-\overline{X}}{( X'' - E[X]\Ro X')*\G}
\\
	&= 0 + 2A^{-1}[X] \IPh{\overline{X}}{( X'' - E[X]\Ro X')*\G}
\\
	&= 2A^{-1}[X] \IPl{\overline{X}}{((X' - E[X]\Ro X)*\G)_u}
	= 0\,.
\end{align*}
\end{proof}

Next, we establish a lower bound for the Dirichlet energy $Q[X(\cdot,t)]$ along
the flow.

\begin{proposition}
\label{PNmiracle4}
Let $X:\S\times[0,T)\rightarrow\R^2$ be a $H^1(du)$-gradient flow for $E$ on
$H^1_+(du)$.
Then 
\[
Q[X(\cdot,t)] \ge \frac{c_0}{4}
\]
where $c_0 = ||X(\cdot,0)-\overline{X}(0)||_{H^1(du)}^2 > 0$.
\end{proposition}
\begin{proof}
First we note that $||X(\cdot,0)-\overline{X}(0)||_{H^1(du)} = 0$ implies that
$X(\cdot,0)$ is a constant, which would mean that $A[X(\cdot,0)] = 0$.
This is impossible since $X(\cdot,0)\in H^1_+(du)$.

Now, using the Poincar\'e inequality with Proposition \ref{PNmiracle3} gives
\begin{align*}
c_0 &= ||X(\cdot,t)-\overline{X}(t)||_{H^1(du)}^2
\\
	&= 
		||X(\cdot,t)-\overline{X}(t)||_{L^2(du)}^2
		+ 2Q[X(\cdot,t)]
\\
	&\le 
		4Q[X(\cdot,t)]
\end{align*}
as required.
\end{proof}

Proposition \ref{PNmiracle4} allows us to conclude uniform estimates for area.
\begin{corollary}
\label{CYmiracle}
Let $X:\S\times[0,T)\rightarrow\R^2$ be a $H^1(du)$-gradient flow for $E$ on
$H^1_+(du)$.
Then 
\[
	\frac{c_0}{4} E^{-1}[X(\cdot,0)]
	\le A[X(\cdot,t)]
	\le \frac12\h{X(\cdot,0)}^2
\]
where $c_0 = ||X(\cdot,0)-\overline{X}(0)||_{H^1(du)} > 0$.
\end{corollary}
\begin{proof}
The monotonicity of the energy and Proposition \ref{PNmiracle4} implies
\[
	E[X(\cdot,0)] \ge E[X(\cdot,t)] = \frac{Q[X(\cdot,t)]}{A[X(\cdot,t)]}
	\ge \frac{c_0}{4} A^{-1}[X(\cdot,t)]
\]
or
\[
	A[X(\cdot,t)] \ge \frac{c_0}{4} E^{-1}[X(\cdot,0)]
	\,.
\]
This is the claimed lower bound.
For the upper bound, use \eqref{EQareawelldef} and Proposition \ref{PNmiracle}
to find
\[
	A[X(\cdot,t)] \le \frac12\h{X(\cdot,t)}^2 = \frac12\h{X(\cdot,0)}^2
	\,.
\]
This finishes the proof.
\end{proof}

A consequence of the lower bound for area (Corollary~\ref{CYmiracle}) and
preservation of norm (Proposition~\ref{PNmiracle}) is global existence:

\begin{corollary}
\label{CYeternal}
For each $X_0\in H^1_+(du)$ there exists a unique map $X\in
C^1(\R; H^1_+(du))$ such that $X(\cdot,0) = X_0$ and
$X:\S\times\R\rightarrow\R^2$ is a $H^1(du)$-gradient flow for $E$ on
$H^1_+(du)$.
\end{corollary}
\begin{proof}
According to Theorem \ref{TMSTE}, given $X_0\in H^1_+(du)$, there is a unique
solution $X(t,u)$ for $t\in [-T_0,T_0]$ provided $KT_0 < b$, where $b$ is the
radius of the cylinder in Lemma \ref{blip}.

Following the standard continuation procedure, one takes
$X_{\pm}:=X(\cdot,\pm T_0)$ as the initial data for a new application of
Theorem \ref{TMSTE}.
The existence time will be extended to $[-T_0-T_1, T_0+T_1]$ provided $KT_1 <
b$, where again $b$ and $K$ are as in Lemma \ref{blip}, but now with $X_0 =
X_{\pm}$.

In the forward and backward time direction, the $H^1(du)$ norm of $X$ is constant.
This means that in order to conclude a uniform amount of time will be added
with each iteration and we will eventually obtain a global solution, an
estimate for the area $A[X(\cdot,t)]$ is required.
This is guaranteed by Corollary \ref{CYmiracle}, and so we are able to add a
uniform amount of time with each iteration of the Picard theorem, and thus
obtain as claimed a global solution.
\end{proof}

\begin{proposition}
\label{PNconv1}
Let $X:\S\times\R\rightarrow\R^2$ be a $H^1(du)$-gradient flow for $E$ on
$H^1_+(du)$.
Suppose that $||X(\cdot,0)-\overline{X}(0)||_{H^1(du)} = c_0 > 0$.
Then there exists a smooth non-constant limit $Y\in C^\infty_+$ (here
$C^\infty_+$ is the subspace of $H^1_+(du)$ consisting of smooth maps) and
sequence $t_n\rightarrow\infty$ such that $X(\cdot,t_n)$ converges (strongly) to $Y$ in
the $C^{0,\alpha}$ topology for $\alpha\in(0,\frac12)$.
Furthermore $Y\in\sS$; in particular, it is a multiply-covered circle.
\end{proposition}
\begin{proof}
Since $\int_0^\infty ||DE[X(\cdot,t)||_{H^1(du)}^2\,dt \le E[X_0] < \infty$,
we have that $t\mapsto ||DE[X(\cdot,t)||_{H^1(du)}^2$ is an $L^1(0,\infty)$ function, and so by Lesigne's theorem (see \cite{L}) for almost every $t>0$ we have
\begin{equation}
\label{EQlesigne}
	\lim_{n\rightarrow\infty} ||DE[X(\cdot,nt)||_{H^1(du)}^2 = 0\,.
\end{equation}
In particular, there exists a sequence $\{t_n\}$ such that $t_n\nearrow\infty$
and 
\[
\lim_{n\rightarrow\infty} ||DE[X(\cdot,t_n)||_{H^1(du)}^2 = 0.
\]

The constancy of the $H^1(du)$ norm implies uniform boundedness of the family
$\{X(\cdot,t_n)\}$ in $H^1(du)$, hence by weak compactness existence of a map
$Y\in H^1(du)$ and subsequence (that we still denote by $t_n$)
$\{t_n\}\subset(0,\infty)$ such that $X(\cdot,t_n)$ converges weakly to $Y$ in
$H^1(du)$.
Rellich--Kondrachov implies that the convergence is strong in $C^{0,\alpha}$ for
any $\alpha\in(0,\frac12)$.


Due to \eqref{EQlesigne}, the limit $Y$ satisfies the equilibrium equation
$DE[Y] = 0$.
Thus we obtain that $Y$ is smooth, and furthermore that $Y$ is not a constant
map.
Proposition \ref{PNequil} applies to give that $Y\in\sS$, in particular, $Y$ is
an $\ell$-times covered circle.
This finishes the proof.
\end{proof}

\begin{remark}
Proposition \ref{PNequil} applies to give not only that $Y\in\sS$ but that $Y$
has parametrisation
\[
		Y(u) = \frac{a}{\ell}\begin{pmatrix} \sin(\ell u) \\ -\cos(\ell u) \end{pmatrix}
 + \frac{b}{\ell}\begin{pmatrix} \cos(\ell u) \\  \sin(\ell u) \end{pmatrix}
	+ \begin{pmatrix} c \\ d \end{pmatrix}, 
\]
for some $a,b,c,d\in\R$ and $\ell\in\N$.
Using the initial data, we may eliminate three of these degrees of freedom and
constrain a fourth.

First, as the centre of the flow is invariant, we have
\[
	\begin{pmatrix} c \\ d \end{pmatrix}
	= \overline{X}(0)
	\,.
\]
This removes two of the constants.
Then, since
\[
	\h{Y - \overline{Y}}^2
	= 2\pi\frac{a^2 + b^2}{\ell} + 2\pi(a^2+b^2)
	= 2\pi\frac{1+\ell}{\ell}(a^2 + b^2)
\]
we find
\[
	a^2+b^2 = \frac{\ell}{2\pi(1+\ell)}\h{X(\cdot,0)-\overline{X}(0)}^2
	\,.
\]
This means that $b$ can be expressed in terms of $a$, $\ell$ and $X(\cdot,0)$.
Finally, we also know that $1 \le \ell \le E[X_0]$, removing one more constant
and constraining $\ell$.
\end{remark}

Next we upgrade the convergence topology to the best possible using a direct argument.
The below is the same as Proposition \ref{PNconv1}, except the convergence
sense is improved to $H^1(du)$.

\begin{proposition}
\label{PNconv2}
Let $X:\S\times\R\rightarrow\R^2$ be a $H^1(du)$-gradient flow for $E$ on
$H^1_+(du)$.
Suppose that $||X(\cdot,0)-\overline{X}(0)||_{H^1(du)} = c_0 > 0$.
Then there exists a smooth non-constant limit $Y\in C^\infty_+$ 
and sequence $t_n\rightarrow\infty$ such that $X(\cdot,t_n)$ converges to $Y$
in the $H^1(du)$ topology; that is,
\begin{equation}
\label{EQbettertop3}
	\lim_{n\rightarrow\infty} \h{X(\cdot,t_n) - Y} = 0\,.
\end{equation}
The limit $Y$ is a member of $\sS$, the stationary set; in
particular, it is a multiply-covered circle.
\end{proposition}
\begin{proof}
We calculate using Lemma \ref{LMfundiden}: 
\begin{equation}
\label{EQbettertop}
\begin{aligned}
	||&X(\cdot,t_n) - Y||_{H^1(du)}^2 \\
	&= 
		||X(\cdot,t_n) - Y||_{L^2(du)}^2
		+ ||X_u(\cdot,t_n) - Y_u||_{L^2(du)}^2 \\
	&= 
		||X(\cdot,t_n) - Y||_{L^2(du)}^2
		- \IPh{(X'(\cdot,t_n)-Y')*\G}{X_u(\cdot,t_n)-Y_u} \\
	&= 
		||X(\cdot,t_n) - Y||_{L^2(du)}^2
		+ \IPh{(X''(\cdot,t_n)-Y'')*\G}{X(\cdot,t_n)-Y}\,. 
\end{aligned}
\end{equation}


Our goal is to take a limit as $n\rightarrow\infty$ in \eqref{EQbettertop}.
We work with the second term on the RHS.
Substituting the highest order terms for the gradient with
\[
	X''*\G = A[X]DE[X] + E[X]\Ro X'*\G
\]
and estimating yields 
\begin{align*}
	&\IPh{(X''(\cdot,t_n)-Y'')*\G}{X(\cdot,t_n)-Y}
\\
	&\qquad = 
		\IPh{A[X(\cdot,t_n)]DE[X(\cdot,t_n)] + E[X(\cdot,t_n)] \Ro X'(\cdot,t_n)*\G - Y''*\G}{X(\cdot,t_n)-Y}
\\
	&\qquad= 
		A[X(\cdot,t_n)]\IPh{DE[X(\cdot,t_n)]}{X(\cdot,t_n)-Y}
	\\&\qquad\qquad
		+ \IPh{E[X(\cdot,t_n)]\Ro X'(\cdot,t_n)*\G- E[Y]{\Ro Y'*\G}}{X(\cdot,t_n)-Y}
\\
	&\qquad\le 
		\frac14\h{X(\cdot,t_n) - Y}^2
		 + A^2[X(\cdot,t_n)]\h{DE[X(\cdot,t_n)]}^2
	\\&\qquad\qquad
		+ \frac14 ||X(\cdot,t_n)-Y||_{H^1(du)}^2
		+ \h{E[X(\cdot,t_n)]\Ro X'(\cdot,t_n)*\G - E[Y]\Ro Y'*\G}^2
\,.
\end{align*}
Combining the above estimate with \eqref{EQbettertop} and absorbing, we find
\begin{align}
	||X(\cdot,t_n) - Y||_{H^1(du)}^2
	&\le 
		2||X(\cdot,t_n) - Y||_{L^2(du)}^2
		 + 2A^2[X(\cdot,t_n)]\h{DE[X(\cdot,t_n)]}^2
\notag	\\&\qquad
		 2\h{E[X(\cdot,t_n)]\Ro X'(\cdot,t_n)*\G - E[Y]\Ro Y'*\G}^2
\label{EQbettertop2}
\,.
\end{align}
The $C^{0,\alpha}$ subconvergence implies the first term on the RHS above
tends to zero as $n\rightarrow\infty$.
For the second, we use the gradient structure and take a further subsequence
(also denoted $t_n$) to ensure that it also vanishes as $n\rightarrow\infty$.
For the third, we calculate further
\begin{align*}
2&\h{E[X(\cdot,t_n)]\Ro X'(\cdot,t_n)*\G - E[Y]\Ro Y'*\G}^2
\\
&= 2\l{E[X(\cdot,t_n)]\Ro X'(\cdot,t_n)*\G - E[Y]\Ro Y'*\G}^2
\\&\quad
+ 2\l{E[X(\cdot,t_n)]\Ro X'(\cdot,t_n)*\G_u - E[Y]\Ro Y'*\G_u}^2
\\
&= 2\l{E[X(\cdot,t_n)]\Ro X'(\cdot,t_n)*\G - E[Y]\Ro Y'*\G}^2
\\&\quad
+ 2||E[X(\cdot,t_n)]\Ro X(\cdot,t_n) + E[X(\cdot,t_n)]\Ro X(\cdot,t_n)*\G
\\&\qquad\qquad\qquad\qquad\qquad\qquad - E[Y]\Ro Y - E[Y]\Ro Y*\G||^2_{L^2(du)}
\\
&\le 2\l{E[X(\cdot,t_n)]\Ro X(\cdot,t_n)*\G_u - E[Y]\Ro Y*\G_u}^2
\\&\quad
+ 4\l{E[X(\cdot,t_n)]\Ro X(\cdot,t_n) - E[Y]\Ro Y}^2
\\&\quad
+ 4\l{E[X(\cdot,t_n)]\Ro X(\cdot,t_n)*\G - E[Y]\Ro Y*\G}^2
\\
&\le \pi^2\l{E[X(\cdot,t_n)]X(\cdot,t_n) - E[Y]Y}^2
\\&\quad
+ 4\l{E[X(\cdot,t_n)] X(\cdot,t_n) - E[Y] Y}^2
\\&\quad
+ 4\pi\coth\pi\l{E[X(\cdot,t_n)] X(\cdot,t_n) - E[Y] Y}^2
\,.
\end{align*}
In the last step above we used the estimates we have on $\G$, namely $|\G_u|
\le \frac12$ and $|\G| \le \frac12\coth\pi$.

Since $E[X(\cdot,t_n)]\rightarrow E[Y]$, convergence in $L^2(du)$ (already
implied by the $C^{0,\frac12}$ convergence) gives that all terms on the RHS
above tend to zero.
This in turn implies all terms on the RHS of \eqref{EQbettertop2} converge to
zero, implying \eqref{EQbettertop3} and finishing the proof.
\end{proof}


\section{Full convergence}

In this section we show how Feehan's new Morse--Bott condition
\cite{feehan2020morse} can be used for our flow to obtain the
 \L ojasiewicz--Simon gradient inequality with \emph{optimal} exponent.

The following definitions are from \cite{feehan2020morse} (1.8 and 1.9) but we
make some small changes to the notation. 

\begin{definition}
Let $\mathcal X$ and $\tilde{\mathcal X}$ be Banach spaces with
$\tilde{\mathcal X}$ continuously embedded $\tilde{\mathcal {X}}\subset
\mathcal X^*$, and $\mathcal U\subset \mathcal X$ an open subset. 
Let $f:\mathcal U\to \R$ be a $C^1$ function.
A continuous map $\mathcal M :\mathcal U\to
\tilde{\mathcal X}$ satisying
\[ df(x)v=(\mathcal M (x),v)_{\mathcal X^*\times \mathcal X}\]
 is called a \emph{gradient map} for $f$. 
\end{definition}

\begin{definition}\label{def:gmb}
Let $\mathcal X, \tilde{\mathcal X}, \mathcal U$ be as above and furthermore
suppose $f:\mathcal U\to \R$ is $C^2$ and $\mathcal M :\mathcal U\to
\tilde{\mathcal X}$ is a gradient map for $f$. If $\crit f$ is a $C^2$
connected submanifold of $\mathcal U$  we say $(f,\mathcal M) $ is Morse--Bott
at $Y\in \crit f$ if 
$T_{Y} \crit f=\ker d\mathcal{M}(Y)$ and  $\ran d\mathcal{M}(Y)=\tilde{\mathcal X}$. 
\end{definition}

\begin{theorem}\label{feehanthm}(\cite{feehan2020morse} Theorem 7)
 Let $\mathcal X, \tilde{\mathcal X}, \mathcal U$ be as above and $f:\mathcal
U\to \R$ a $C^{3}$ function 
 such that $f(Y)=0$ and $f'(Y)=0$, and
let $\mathcal M:\mathcal X\to \tilde{\mathcal X}$ be a gradient map for
$f$. If $(f,\mathcal M)$ is Morse--Bott at $Y$ in the sense of
Definition \ref{def:gmb} then, after possibly shrinking $\mathcal U$, there is
a constant $C\in (0,\infty)$ such that
\begin{equation}
\norm{\mathcal M(x)}_{\tilde {\mathcal X}}\geq C\abs{f(x)}^{1/2}, \quad \text{ for all } x\in \mathcal U
\,.
\end{equation}
\end{theorem}

In our case we have $f = E - E[Y]:\SU\rightarrow\R$, $f\in C^3$, $\SU$ an open
neighbourhood of $Y$, and $\mathcal M:= DE:\SU\to H^1(du)$ given by
\[
DE[X]=A^{-1}[X](X''-E[X] \rot X') * \G
\]
and then 
\begin{align*}
d\M[X]W
&=
	- A^{-2}[X] dA[X]W\, DE[X]
\\&\qquad
	+ A^{-1}[X](W''-dE[X]W \rot X'-E[X] \rot W') * \G
\,.
\end{align*}
If $Y$ is a critical point then $DE[Y]=0=dE[Y]$ so 
\[ 
d\M[Y]W=A^{-1}[Y](W''-E[Y]\rot W')*\G\,.
\]

\begin{theorem}
\label{TMmb}
With notation as above, the pair $(E,DE)$ is Morse--Bott at $Y$.
\end{theorem}
\begin{proof}
We need to show that $0$ is a regular value of $DE=\M$, that is, $d\M[Y]$ is
surjective and its kernel splits.
Given $Y\in H^1(du)$ we solve $d\M[Y]W=Y$ as follows. First note that 
\begin{alignat*}{2}
&&\qquad A^{-1}[Y](W''-E[Y]\rot W') * \G &= Y \\
&\iff&
 \ip{A^{-1}[Y](W''-E[Y]\rot W') * \G,\phi}_{H^1(du)}
	&= \ip{Y,\phi}_{H^1(du)}\quad \text{ for all } \phi\in H^1(du)\\
&\iff& \ip{A^{-1}[Y](W'-E[Y]\rot W),\phi_u}_{L^2(du)}
	&= \ip{\mathcal Y,\phi_u}_{L^2(du)}
\end{alignat*} 
where $\mathcal Y := Y_u - \int_0^u Y(\tau)d\tau $.
Therefore we will have shown surjectivity if we solve 
\begin{equation*}
W_u-E[Y]\rot W=A[X_0]\mathcal Y
\end{equation*}
or equivalently (cf. Section 4)
\[
W_u=\begin{pmatrix}
0& -\ell \\
\ell & 0
\end{pmatrix}W+A[X_0]\mathcal Y
\]
where $\ell:=E[Y]$ is a positive integer. 
The solution to this with initial condition $W_0$ is
\[
W(u)=e^{Bu}(W_0+B^{-1}A[X_0]\mathcal Y(u))-B^{-1} A[X_0]\mathcal Y(u), \quad B=\begin{pmatrix}
0& -\ell \\
\ell & 0
\end{pmatrix}.
\]
Noting that $e^{Bu}=\begin{pmatrix}
\cos(\ell u)& -\sin (\ell u) \\
\sin (\ell u) & \cos(\ell u)
\end{pmatrix}$ we find that $W(2\pi)=W(0)$, and therefore $W\in H^1(du)$.

The kernel of $d\M[Y]$ splits because it is a continuous linear map and so its
kernel is a closed subspace of a Hilbert space. It then follows from the regular
values theorem that $\crit E$ is a locally connected submanifold and in
particular there is a neighbourhood $\SU$ of $Y$ that is connected.
Then the tangent space $T_{Y}\crit E=\ker dF[Y]$, and so the pair $(E, DE)$
is indeed Morse--Bott at $Y$.
\end{proof}

We therefore conclude the gradient inequality.

\begin{corollary}(Gradient inequality)\label{gradineq}
If $Y$ is a critical point of $E$ then there is a neighbourhood
$\mathcal U\subset H^1(du)$ containing $Y$ and a constant $C\in (0,\infty)$
such that 
\begin{equation}\label{gi}
\norm{DE[X]}_{H^1(du)}\geq C(E[X]-E[Y])^{1/2}, \quad \text{for all } X\in \mathcal U 
\,.
\end{equation} 
\end{corollary}
\begin{proof}
Theorem \ref{TMmb} implies that $(E,DE)$ is Morse--Bott at $Y$, and so the
result follows from application of Feehan's Theorem \ref{feehanthm} in \cite{feehan2020morse}.
\end{proof}

\begin{remark}
Since critical points are local minimisers of the energy, the difference
$E[X]-E[Y]$ is non-negative for $\mathcal U$ small.
\end{remark}

Using the gradient inequality we obtain full convergence by a standard argument.
For the convenience of the reader, we include a proof below.

\begin{proposition}
Let $X:\S\times\R\rightarrow\R^2$ be a $H^1(du)$-gradient flow for $E$ on
$H^1_+(du)$.
Then there is an equilibria $Y\in\sS$ such that
$\norm{X(\cdot,t)-Y}_{H^1(du)}\to
0$ as $t\to \infty$.
\label{PNconv}
\end{proposition}
\begin{proof}
By Proposition \ref{PNconv2} there is a sequence of times $(t_i)$ and a
stationary point $Y$ such that $\norm{X(t_i)-Y}_{H^1}\to 0$, and then by
Corollary \ref{gradineq} there is a subsequence, still denoted $(t_i)$, such
that every $X(\cdot,t_i)$ is contained in a neighbourhood $\mathcal U$ of $Y$
where \eqref{gi} holds for all $X\in \mathcal U$. 
For each $t_i$, let $T_i=\sup\{t>t_i: X(\cdot,t)\in \mathcal U\}$, and let
$I:=\cup_i[t_i,T_i).$

Define $\Phi(t):=(E[X(\cdot,t)]-E[Y])^{1/2}$.
Note that $\Phi$ is decreasing and has derivative
\begin{align*}
\Phi'(t)&=\frac{1}{2}(E[X(\cdot,t)]-E[Y])^{-1/2}dE[X(\cdot,t)]X_t(\cdot,t)\\
&=-\frac{1}{2}(E[X(\cdot,t)]-E[Y])^{-1/2}\norm{DE[X(\cdot,t)]}_{H^1(du)}^2
\,.
\end{align*}
Then for all $t\in I$, since \eqref{gi} holds, 
\begin{align*}
\Phi'(t)
&\le - \frac{1}{2}C\norm{DE[X(\cdot,t)]}_{H^1(du)}^{-1}\norm{DE[X(\cdot,t)]}_{H^1(du)}^2
\intertext{or}
 \frac{C}{2}\norm{DE[X(\cdot,t)]}_{H^1(du)} &\leq -\Phi'(t)
\,.
\end{align*}
Now integrating over $I$ and using the fact that $\Phi$ is decreasing:
\begin{equation}\label{intgrad}
\int_I \norm{DE[X(\cdot,t)]}_{H^1(du)}\, dt
	\leq 2C^{-1}\Phi(t_0)\leq 2C^{-1}E[X(\cdot,0)]^{\frac12}\, .
\end{equation}
In fact, there exists $t_N$ such that $T_N=\infty$.
If not, choosing $\delta$ such that the $H^1(du)$ ball $B_{\delta}(Y)$ of radius
$\delta$ centred at $Y$ is contained in $\mathcal U$, then for each $i$ there
exists $\tau_i$ such that $X(\cdot,\tau_i)$ is on the boundary of $B_{\delta}(Y)$.
But then
\begin{align*}
\delta = \norm{X(\cdot,\tau_i)-Y}_{H^1(du)}
&\leq \norm{X(\cdot,t_i)-Y}_{H^1(du)} + \norm{X(\cdot,t_i)-X(\cdot,\tau_i)}_{H^1(du)} 
\\
&\leq \norm{X(\cdot,t_i)-Y}_{H^1(du)} + \int_{t_i}^{\tau_i}\h{DE[X(\cdot,t)]} \, dt
\end{align*}
which contradicts the boundedness of the integral in \eqref{intgrad}.
Therefore such a $t_N$ exists and then from \eqref{intgrad}
$\int_{t_N}^\infty\norm{X_t(\cdot,t)}_{H^1(du)}\, dt < \infty$, that is the length of the
gradient trajectory is finite, and so it converges. Indeed if we have any
sequence of times $t_i$ then $(X(\cdot,t_i))$ is Cauchy, because if not then
$\norm{X(\cdot,t_i)-X(t_j)}\leq \int_{t_i}^{t_j}\norm{X_t} dt$ gives a
contradiction to \eqref{intgrad}.
\end{proof}

We also obtain an exponential rate.

\begin{corollary}
\label{CYrate}
Let $X:\S\times\R\rightarrow\R^2$ be a $H^1(du)$-gradient flow for $E$ on
$H^1_+(du)$.
Then there is a multiply-covered circle $Y\in\sS$ and constants $c,C>0$ such that
\[
\h{X(\cdot,t)-Y} \le Ce^{-ct}\,.
\]
In other words, the convergence rate is exponential.
\end{corollary}
\begin{proof}
First, note that the function $\Phi$ from Proposition \ref{PNconv} satisfies, by the gradient inequality,
\begin{equation}
\label{EQrate1}
	\Phi'(t) \le -\frac{C}{2}\Phi(t)
\end{equation}
so
\begin{equation}
\label{EQrate2}
	\Phi(t) \le Ce^{-ct}
\end{equation}
for constants $c,C$.

Now from the evolution equation
\[
	X_t = -DE[X]
\]
we find
\[
	Y - X(\cdot,t) = \int_t^\infty X_{\hat t}(\cdot,\hat t)\,d\hat t
		= -\int_t^\infty DE[X(\cdot, \hat t)]\,d\hat t\,.
\]
Thus, using the notation of Proposition \ref{PNconv}
\[
	\h{Y - X(\cdot,t)}
	\le C\int_t^\infty |\Phi'(\hat t)|\,d\hat t\,.
\]
Using equations \eqref{EQrate1} and \eqref{EQrate2} we thus find (here $C$
varies from line to line but remains universal)
\begin{align*}
	\h{Y - X(\cdot,t)}
	&\le C\int_t^\infty |\Phi(\hat t)|\,d\hat t
\\
	&\le C\int_t^\infty e^{-c\hat t}\,d\hat t
\\
	&\le Ce^{-c\hat t}
\end{align*}
as required.
\end{proof}


\section{Isoperimetric inequalities}

Finally we present the application to isoperimetry.

\begin{theorem}
\label{TM1}
Suppose $X\in H^1(du)$.
Then the isoperimetric inequality holds:
\begin{equation}
\label{EQisoi}
	\I[X] \ge 1\,.
\end{equation}
\end{theorem}
\begin{proof}
First, if $A[X] = 0$, the isoperimetric inequality \eqref{EQisoi} trivially
holds.  
So we may assume that $A[X] \ne 0$.
If $A[X] < 0$, then reparametrise $X$ by sending $u$ to $2\pi-u$.
This leaves the isoperimetric ratio invariant as there is an absolute value
sign around $A[X]$, and makes $A[X]$ positive.

Second, we reparametrise $X(\S)$ using a map $X_0\in H^1(du)$ with constant
speed.
Note that we can do this so long as $X\in W^{1,1}(du)$, so since $X\in H^1(du)$
this is no problem.
Let us consider $X_0 = X$ as initial data for a $H^1(du)$ gradient flow for
$E$.
Both of these preparations do not change the set $X_0(\S)$, its length or the
magnitude of the area.

Now, since $|(X_0)_u| = L[X_0]/2\pi$, we have
\[
Q[X_0] = \frac12 \int_0^{2\pi} \frac{L^2[X_0]}{4\pi^2}\,du
= \frac{L^2[X_0]}{4\pi}
\,.
\]
Thus, taking the $H^1(du)$-gradient flow of $E$ with $X_0\in H^1_+(du)$ as initial data,
\begin{align*}
\I[X_0]
	&= \frac{L^2[X_0]}{4\pi A[X_0]}
\\
	&= E[X_0]
\\
	&\ge \lim_{t\rightarrow\infty} E[X(\cdot,t)]
\\
	&= E\Big[\lim_{t\rightarrow\infty}X(\cdot,t)\Big]
\\
	&= E[Y]
	\ge \I[Y]\,.
\end{align*}
The second last equality uses the continuity of $E$ on $H^1(du)$ and the convergence result Proposition~\ref{PNconv}.

Proposition \ref{PNequil} implies $Y$ is an $\ell$-circle, where $\ell$ is a positive integer smaller than $E[X_0]$, and we have
\[
	\I[Y] = \ell\,.
\]
Therefore we obtain $\I[X_0] \ge \ell \ge 1$, the isoperimetric inequality.
\end{proof}

\begin{remark}
Convergence can not be in any higher topology, and indeed we must
have ``non-symmetric'' loops shrinking to zero under the flow.
An interesting point is to study what happens for a symmetric lemniscate, with $A[X_0]=0$.
We conjecture that it shrinks to a point in finite time, as do circles with negative area.
However the lemniscate is more difficult, since its energy is equal to
$-\infty$ along the entire evolution.
\end{remark}

\begin{theorem}
Suppose $X\in H^1(du)$ is a rotationally symmetric map in the sense that there
are positive integers $n, m$ such that $1\le n \le m$ and
\begin{equation}
\label{EQsym}
	X\Big(u+\frac{2\pi}{m}\Big) = \text{Rot}{\,}_{\frac{2\pi n}{m}} X(u)\,.
\end{equation}
Then we have
\[
	\I[X] \ge n\,.
\]
\end{theorem}
\begin{proof}
Following the same procedure as in the proof of Theorem \ref{TM1} we see that
$\I[X] \ge \ell$ where $\ell$ is the number of times the limit circle $Y$ has been
covered.

The symmetry condition \eqref{EQsym} is preserved into the limit $Y$.
This is due to invariance and uniqueness.
Both the $H^1(du)$ metric and the
energy functional $E$ are invariant under rotations and shifts in the $u$ parameter.  Therefore (cf. \cite[Lemma 3.1]{SWW}) if $X(u,t)$ is a $H^1(du)$-gradient flow for $E$ then so are $X(u+2\pi/m,t)$ and $Rot_{2\pi n/m}X(u,t)$. These two flows coincide at $t=0$, and so by uniqueness (Corollary \ref{CYeternal}) they coincide for all $t$, i.e. $X(u+2\pi/m,t)=Rot_{2\pi n/m}X(u,t)$.
Since the limit $Y$ is unique by Proposition \ref{PNconv}, it must also satisfy \eqref{EQsym}.

So, we find that $Y$ is a circle in the set $\sS$ that additionally satisfies
the symmetry condition \eqref{EQsym}.
This means that on one period, for example $(0,2\pi/m)$, $Y$ is a circular arc.
By rotating if necessary (translation is not needed as the centre of $Y$ must
be the origin already by the rotational symmetry; note that rotation does not
affect either the number $\ell$ or the validiting of \eqref{EQsym}), we may
take the parametrisation $u \mapsto \frac{a}{\ell}(\cos \ell u, \sin \ell u)$
for $Y$.
Since $Y$ satisfies \eqref{EQsym}, we have
\[
\text{Rot}{\,}_{\frac{2\pi n}{m}} Y(0)
=
\begin{pmatrix}
\cos \frac{2n\pi}{m} & -\sin \frac{2n\pi}{m} \\
\sin \frac{2n\pi}{m} &  \cos \frac{2n\pi}{m}
\end{pmatrix}
Y(0)
=
Y\Big(\frac{2\pi}{m}\Big)
\]
which implies
\[
\begin{pmatrix}
\cos \frac{2n\pi}{m} & -\sin \frac{2n\pi}{m} \\
\sin \frac{2n\pi}{m} &  \cos \frac{2n\pi}{m}
\end{pmatrix}
\begin{pmatrix}
1 \\ 0
\end{pmatrix}
=
\begin{pmatrix}
\cos \frac{2n\pi}{m}
\\
\sin \frac{2n\pi}{m}
\end{pmatrix}
=
\begin{pmatrix}
\cos \frac{2\ell\pi}{m}
\\
\sin \frac{2\ell\pi}{m}
\end{pmatrix}
\,.
\]
We conclude that 
$\cos \frac{2\ell\pi}{m} = \cos \frac{2n\pi}{m}$ and $\sin \frac{2n\pi}{m} = \sin \frac{2\ell\pi}{m}$, which means that
\[
\frac{2\ell\pi}{m} = \frac{2n\pi}{m} + 2k\pi\,,\qquad k\in\Z
\]
or
\[
\ell = n + km\,.
\]
From Proposition \ref{PNequil} we know that $\ell\ge1$.
Since $n\le m$ this implies that $k \ge 0$.
Therefore $\ell \ge n$, which implies 
$\I[X] \ge l \ge n$ and we are finished.
\end{proof}


\bibliographystyle{plain}
\bibliography{gradient_flows}

\end{document}